\documentclass[12pt]{article}
\usepackage{enumerate}
\usepackage{amsmath,amssymb,amsthm, mathrsfs, mathtools}
\usepackage{latexsym}
\usepackage{graphics,graphicx}
\usepackage{hyperref}
\usepackage[bf, small]{titlesec}
\usepackage{authblk} 
\usepackage{soul}
\usepackage{kotex}
\usepackage{lineno}

\theoremstyle{plain}
\newtheorem{Thm}{Theorem}[section]
\newtheorem{XThm}{Theorem}
\newtheorem{Lem}[Thm]{Lemma}
\newtheorem{Prop}[Thm]{Proposition}
\newtheorem{Cor}[Thm]{Corollary}

\theoremstyle{definition}

\newtheorem{Rmk}[Thm]{Remark}

\usepackage{standalone}
\usepackage{tikz}
\usetikzlibrary{arrows,positioning,matrix,fit,backgrounds,shapes,shapes.geometric, intersections}
\usetikzlibrary{decorations.markings}
\tikzstyle{vertex}=[circle, draw, inner sep=0pt, minimum size=6pt] 
\usetikzlibrary{shapes.callouts,decorations.pathmorphing, calc}
\usetikzlibrary{decorations.markings} 
\tikzstyle{vertex}=[circle, draw, inner sep=0pt, minimum size=6pt] 
\newcommand{\vertex}{\node[vertex]}
\usetikzlibrary{arrows,matrix} 
\title{Multipartite tournaments in which any two vertices have an $(i,j)$-step common out-neighbor}

\author[1]{\small Myungho Choi}
\author[1]{\small Suh-Ryung Kim}

\affil[1]{\footnotesize Department of Mathematics Education, Seoul National University, Seoul 08826}
\affil[ ]{\footnotesize\textit{nums8080@snu.ac.kr,  srkim@snu.ac.kr}}
\date{}
 \newcounter{statement}

\begin{document}
\maketitle
\begin{abstract}
     We say that a digraph $D$ is $(i,j)$-step competitive if any two vertices have an $(i,j)$-step common out-neighbor in $D$ and that a graph $G$ is $(i,j)$-step competitively orientable if there exists an $(i,j)$-step competitive orientation of $G$. 
     
 In [Choi et al. Competitively orientable complete multipartite graphs. Discrete Mathematics, 345(9):112950, 2022],
    Choi et al. introduce the notion of competitive digraph and completely characterize competitively orientable complete multipartite graphs in terms of the sizes of its partite sets.
    Here, a competitive digraph means a $(1,1)$-step competitive digraph.
     In this paper, the result of Choi et al. has been extended to a general characterization of $(i,j)$-step competitively orientable complete multipartite graphs.
         \end{abstract}

    \noindent
{\it Keywords.} $(i,j)$-step competitive digraph; $(i,j)$-step competitively orientable graph; complete multipartite graph; multipartite tournament;  $(i,j)$-step competition graph

\noindent
{{{\it 2010 Mathematics Subject Classification.} 05C20, 05C75}}

\section{Introduction}
In this paper, we consider finite graphs without loops and multiple edges and finite digraphs without loops, directed $2$-cycles, and multiple arcs.
For graph-theoretical terminology and notations not defined, we refer to \cite{bondy2010graph}. 

For a digraph $D$,
the \emph{underlying graph} of $D$
is the graph $G$
such that $V(G)=V(D)$ and $E(G)=\{ uv \mid (u,v) \in A(D) \}$. An \emph{orientation} of a graph $G$
is a digraph whose underlying graph is $G$.
   A \emph{tournament} is an orientation of a complete graph.
A \emph{$k$-partite tournament} is  an orientation of a complete $k$-partite graph for some positive integer $k \geq 2$.
If a digraph is a $k$-partite tournament for some integer $k \geq 2$, then it is called a \emph{multipartite tournament}.
Multipartite tournaments have been actively studied by  graph theorists (see \cite{aboulker2024heroes,figueroa2016strong,fisher1998domination,guo2012weakly,gutin2024note,jung2023competition,zhang2021hamiltonicity} and a survey paper \cite{volkmann2007multipartite}).

   Let $D$ be a digraph. For vertices $u$ and $v$ in $D$, the {\it distance} from $u$ to $v$, denoted by $d_D(u,v)$, is the number of arcs in a shortest directed $(u,v)$-path in $D$.
For positive integers $i$ and $j$, we say that two vertices $u$ and $v$ in $D$ {\it $(i,j)$-step compete}
if there exists a vertex $w$ such that
\begin{itemize}
\item[(i)] $1\leq d_{D-v}(u,w)\leq i$ and $1\leq d_{D-u}(v,w)\leq j$ or 
\item[(ii)] $1\leq d_{D-v}(u,w)\leq j$ and $1\leq d_{D-u}(v,w)\leq i$,
\end{itemize}
equivalently, 
\begin{itemize}
\item[(i)] there exist a directed $(u,w)$-walk of length at most $i$ not traversing $v$ and a directed $(v,w)$-walk of length at most $j$ not traversing $u$ or 
\item[(ii)] there exist a directed $(u,w)$-walk of length at most $j$ not traversing $v$ and a directed $(v,w)$-walk of length at most $i$ not traversing $u$.
\end{itemize}
We call $w$ in the above definition an {\it $(i,j)$-step common out-neighbor} of $u$ and $v$.
 The {\it $(i,j)$-step competition graph} of $D$ is the graph with the vertex set $V(D)$ and an edge $uv$ if and only if $u$ and $v$ $(i,j)$-step compete.
 This notion was introduced by Factor {\it et al.} \cite{factor20111} in 2011 as a generalization of competition graphs, which are $(1,1)$-step competition graphs.
 Research on $(i,j)$-step competition graphs for $(i,j)\neq (1,1)$ is generally difficult, so research has been mainly conducted on $(1,2)$-step competition graphs (see 
\cite{choi20171,factor20111,li20211,zhang20161} for papers related to this topic).

In this paper, we introduce the notion of   $(i,j)$-step competitive orientations of graphs.
We say that a digraph $D$ is {\it $(i,j)$-step competitive} if any two vertices $(i,j)$-step compete in $D$.
It is easy to see that a digraph is $(i,j)$-step competitive if and only if its $(i,j)$-step competition graph is a complete graph.
We call a graph $G$ {\it $(i,j)$-step  competitively orientable}  if there exists an $(i,j)$-step competitive orientation of $G$.

The notions of $(i,j)$-step competitive digraphs and $(i,j)$-step competitive orientations have been derived from the process of investigating the conditions under which $(i,j)$-step competition graphs of multipartite tournaments are complete graphs. Since the concept of $(i,j)$-step competition graphs is a generalization of competition graphs, this paper can be said to be a generalization of \cite{choi2022competitively} on competitively orientable graphs.

Choi \emph{et al.}~\cite{choi2022competitively} completely characterize $(1,1)$-step competitively orientable complete multipartite graphs in terms of the sizes of their partite sets.
Hereby, we only consider $(i,j)$-step competitively orientable complete multipartite graphs when $(i,j)\neq(1,1)$ and characterize them as follows.
\begin{XThm}\label{thm:(1,2)-step-bipartite}
 Let $n_1$ and $n_2$ be positive integers with $n_1 \geq n_2$.
  Then the complete bipartite graph $K_{n_1,n_2}$ is $(i,j)$-step competitively orientable for $(i,j)\neq (1,1)$ if and only if one of the following holds: 
  \begin{enumerate}[(a)]
\item   $i+j=3$, and (i) $n_2 \geq 6$ or (ii) $n_1\geq 10$ and $n_2=5$;
\item $i+j\geq4$ and $n_2\geq 4$;
\item $i\geq 2$, $j\geq 2$, $n_1 \geq 6$, and $n_2= 3$.
\end{enumerate}
\end{XThm}

\begin{XThm} \label{thm:tripartite-(i,j)}
	Let $n_1,n_2$, and $n_3$ be positive integers such that $n_1\geq n_2 \geq  n_3$.
Then the complete tripartite graph $K_{n_1,n_2,n_3}$ is $(i,j)$-step competitively orientable for $(i,j)\neq (1,1)$
 if and only if one of the following holds:
\begin{enumerate}[(a)]
\item $n_3\geq 2$;
\item $n_2 \geq 3$ and $n_3=1$;
\item $i\geq 2$, $j\geq 2$, $n_1 \geq 4$, $n_2=2$, and $n_3=1$.
\end{enumerate}
\end{XThm}

\begin{XThm} \label{thm:(i,j)-step-complete}
Let $k$ be a positive integer with $k \geq 4$ and $n_1,n_2,\ldots$, and $n_k$ be positive integers such that $n_1 \geq \cdots \geq n_k$.
Then the complete $k$-partite graph $K_{n_1,n_2,\ldots,n_k}$ is $(i,j)$-step competitively orientable for $(i,j)\neq (1,1)$
 if and only if one of the following holds:
\begin{itemize}
\item[(a)] $k=4$, and (i) $n_1 \geq 3$ and $n_2=1$ or (ii) $n_2 \geq 2$;
\item[(b)] $k \geq 5$.
\end{itemize}
\end{XThm}

\section{
Preliminaries}\label{sec:pre}

In this section, we investigate the properties of general $(i,j)$-step competitively orientable graphs, not limited to complete multipartite graphs, and these results are used to prove Theorems~\ref{thm:(1,2)-step-bipartite}, \ref{thm:tripartite-(i,j)}, and \ref{thm:(i,j)-step-complete} in Section~\ref{sec:multi}.

By the definition of $(i,j)$-step competitive digraphs,
the following proposition is obviously true.
\begin{Prop}\label{prop:i,j-step_inclusion}
Let $D$ be a digraph and $i$, $i'$, $j$,  and $j'$ be positive integers such that $i\leq i'$ and $j \leq j'$. Then the following are true:
\begin{enumerate}[(1)]
\item if two vertices $u$ and $v$ $(i,j)$-step compete in $D$,
then $u$ and $v$ $(i',j')$-step compete in $D$;
\item if $D$ is $(i,j)$-step competitive,
then $D$ is $(i',j')$-step competitive.
\end{enumerate}

\end{Prop}

\begin{Lem}\label{lem:i,j-step-outdegree}
Let $D$ be a nontrivial $(i,j)$-step competitive digraph. Then each vertex has at least two out-neighbors.
\end{Lem}
\begin{proof}
	Take a vertex $u$ in $D$.
	Since $D$ is nontrivial,
	there exists a vertex $v$ distinct from $u$.
	In addition, since $u$ and $v$ $(i,j)$-step compete,
	$u$ has an out-neighbor $w$.
	If $w$ is the only out-neighbor of $u$,
	then $w$ lies on any directed path from $u$ and so $u$ and $w$ cannot $(i,j)$-step compete, which is impossible.
	Thus $u$ has at least two out-neighbors.
	\end{proof}

\begin{Prop}\label{prop:tournament}
Let $D$ be a nontrivial tournament.
Suppose that $i$ and $j$ are positive integers with $i+j\geq 3$.
Then each vertex in $D$ has at least two out-neighbors if and only if 
$D$ is $(i,j)$-step competitive.
\end{Prop}
\begin{proof}
The ``if" part is an immediate consequence of Lemma~\ref{lem:i,j-step-outdegree}.
To show the ``only if" part, take two vertices $u$ and $v$ in $D$.
If $u$ and $v$ $(1,1)$-step compete, then they $(1,2)$-step compete by Proposition~\ref{prop:i,j-step_inclusion}(1).
Suppose $u$ and $v$ do not $(1,1)$-step compete.
Then, since each of $u$ and $v$ has at least two out-neighbors,
there exists an out-neighbor $u'$ (resp. $v'$) of $u$ (resp. $v$) distinct from $u$ and $v$.
Since $D$ is a tournament,
$u' \to v'$ or $v' \to u'$.
Thus, in each case, $u$ and $v$ $(1,2)$-step compete.
Hence $D$ is $(1,2)$-step competitive and so,
by Proposition~\ref{prop:i,j-step_inclusion}(2),
is $(i,j)$-step competitive.
\end{proof}

\begin{Lem} \label{lem:i,j-step-same-outneighbor}
Let $D$ be a nontrivial $(i,j)$-step competitive digraph and $u$ be a vertex in $D$.
Then the digraph $F$ obtained by adding a vertex $v$ and the arc set \[\{(v,w) \mid \text{$w$ is an out-neighbor of $u$ in $D$}\}\] to $D$ is $(i,j)$-step competitive.
\end{Lem}

\begin{proof}
Obviously, $D$ is a subdigraph of $F$.
To show that $F$ is $(i,j)$-step competitive,
take two vertices $x$ and $y$ in $F$.
If none of $x$ and $y$ equals $v$, then $x$ and $y$ $(i,j)$-step compete in $D$ and so they do the same in $F$.
Now suppose that one of $x$ and $y$ is $v$.
Without loss of generality, we may assume $x=v$.
Suppose $y \neq u$.
Then, since $D$ is $(i,j)$-step competitive,
$u$ and $y$ have an $(i,j)$-step common out-neighbor $z$ in $D$.
By the construction of $F$,
$z$ is an $(i,j$)-step common out-neighbor of $v$ and $y$.
Now we suppose $y=u$.
Then, by Lemma~\ref{lem:i,j-step-outdegree}, $u$ has an out-neighbor $z'$ in $D$.
Thus $z'$ is a common out-neighbor of $u$ and $v$ in $F$ and so, by Proposition~\ref{prop:i,j-step_inclusion}(1), $u$ and $v$ $(i,j)$-step compete.
Therefore
we have shown that $F$ is $(i,j)$-step competitive.\end{proof}

\begin{Thm}\label{thm:supergraph} Let $G$ be an $(i,j)$-step competitively orientable graph and $G'$ be a supergraph of $G$ such that for each vertex $v$ in $G'$, there exists a vertex $u$ in $G$ satisfying $N_G(u)\subseteq N_{G'} (v)$. Then $G'$ is  $(i,j)$-step competitively orientable.	
\end{Thm}
\begin{proof}
	Let $D$ be an $(i,j)$-step competitive orientation of $G$.
If $V(G')= V(G)$, then each orientation $D'$ of $G'$ obtained by orienting edges in $E(G')- E(G)$ arbitrarily so that $A(D) \subseteq A(D')$ is $(i,j)$-step competitive.
Suppose $V(G') \neq V(G)$ and let $V(G')- V(G)=\{v_1,\ldots,v_k\}$ for a positive integer $k$. By the hypothesis, there exists a vertex $u_l$ in $G$ such that $N_G(u_l)\subseteq N_{G'}(v_l)$ for each $1\leq l \leq k $.
Let $G_0=G$, $D_0=D$, and $G_l=G'[V(G) \cup \{v_1,\ldots,v_l\}]$ for each $1\leq l \leq k $.
Then the orientation $D_l$ of $G_l$ obtained by orienting edges in $E(G_l)- E(G_{l-1})$ arbitrarily as long as $A(D_{l-1}) \subseteq A(D_{l})$ and $N^+_{D_{l-1}}(u_{l}) \subseteq N^+ _{D_{l}}(v_{l})$ is $(i,j)$-step competitive for each $1\leq l \leq k $ by Lemma~\ref{lem:i,j-step-same-outneighbor}.
Therefore $D_k$ is an $(i,j)$-step competitive orientation of $G'$.
\end{proof}

Let $G$ be a connected graph.
The {\it diameter} of $G$ is the maximum distance between the pairs of the vertices. 
A {\it disconnecting set} of edges is a set $F \subseteq E(G)$ such that $G-F$ has more than one component.
A graph is {\it $k$-edge connected} if every disconnecting set has at least $k$ edges.

\begin{Thm}\label{lem:(i,j)-char}
Let $G$ be a nontrivial $(i,j)$-step competitively orientable graph.
Then the following are true:
	\begin{enumerate}[(1)]
	\item any vertex in $G$ has degree at least $2$;
 	\item $|V(G)|\geq 5$ and $|E(G)|\geq 2|V(G)|$;
	\item if $G$ has a vertex $u$ of degree $2$, then $G-u$ is $(i,j)$-step competitively orientable; 
	\item for any two vertices $u$ and $v$,
	there exists a $uv$-walk of length at most $i+j$ on which $u$ and $v$ are not consecutive;
	\item the diameter of $G$ is at most $i+j$;
	\item $G$ is $2$-edge connected. 
\end{enumerate}
\end{Thm}

\begin{proof}
Let $D$ be an $(i,j)$-step competitive orientation of $G$.
By Lemma~\ref{lem:i,j-step-outdegree},
each vertex in $D$ has outdegree at least $2$ and so (1) is true.
Thus
$|E(G)|\geq 2|V(G)|$.
Then, since
\[ |E(G)| \leq \frac{|V(G)|(|V(G)|-1)}{2}, \]
$|V(G)|\geq 5$.
Hence (2) is true.

To show (3),
suppose that $G$ has a vertex $u$ of degree $2$.
Then, by Lemma~\ref{lem:i,j-step-outdegree}, 
the two neighbors of $u$ must be out-neighbors in $D$
and so
$u$ has indegree $0$ in $D$.
Thus
the digraph obtained by deleting $u$ from $D$ is an $(i,j)$-step competitive orientation of $G-u$ and so (3) is true.

To show (4),
take two vertices $u$ and $v$ in $G$.
Then, since $D$ is $(i,j)$-step competitive,
there exists an $(i,j)$-step common out-neighbor $w$ of $u$ and $v$.
Without loss of generality, we may assume that
there exist a directed $(u,w)$-path $P_1$ of length at most $i$ not traversing $v$ and a directed $(v,w)$-path $P_2$ of length at most $j$ not traversing $u$ in $D$.
We note that $P_1P_2^{-1}$ is a directed $(u,v)$-walk of length at most $i+j$ not traversing neither $(u,v)$ nor $(v,u)$ where $P_2^{-1}$ is the directed $(w,v)$-path obtained by reversing the sequences of $P_2$.
In addition, $P_1P_2^{-1}$ results in $uv$-walk of length at most $i+j$ on which $u$ and $v$ are not consecutive in $G$.
Since $u$ and $v$ were arbitrarily chosen from $V(G)$, (4) is true.

We can check that (5) is an immediate consequence of (4).

To show (6),
take an edge $uv$ in $G$.
Then, by (4), 
there exists a $uv$-walk of length at most $i+j$ on which $u$ and $v$ are not consecutive in $G$.
In addition, $G$ is connected.
Thus $G-uv$ is connected and so (6) is true.
\end{proof}

\begin{Rmk} \label{rmk:(i,j)step-edges-vertices}
We note that each of the inequalities 
 in (2) of Theorem~\ref{lem:(i,j)-char} is sharp for any positive integers $i$ and $j$ with $i+j\geq 3$.
For, the orientation $D_{10}$ of $K_5$ given in  Figure~\ref{fig:tournaments-complete}  is $(i,j)$-step competitive with $i+j\geq 3$ by Proposition~\ref{prop:tournament}.
Thus the inequalities 
 in (2) of Theorem~\ref{lem:(i,j)-char} are sharp. 
\end{Rmk}

\section{The complete multipartite graphs that are $(i,j)$-step competitively orientable}\label{sec:multi}
In this section,
we completely characterize the $(i,j)$-step competitively orientable complete multipartite graphs in terms of the sizes of their partite sets with $i+j\geq 3$.
Choi \emph{et al.}~\cite{choi2022competitively} completely characterize $(1,1)$-step competitively orientable complete multipartite graphs in terms of the sizes of their partite sets.
Hereby, we only consider $(i,j)$-step competitively orientable complete multipartite graphs when $(i,j)\neq(1,1)$.

Choi \emph{et al.}~\cite{choi20171} did some research on a $(1,2)$-step competitive orientation of a complete bipartite graph and 
the following are some results (the text has been altered in places to fit the terminology of this paper).
\begin{Thm}[\cite{choi20171}]\label{thm:distinct-(1,2)-partite}
Let $u$ and $v$ be vertices belonging to distinct partite sets of an orientation $D$ of a complete bipartite graph.
Then $u$ and $v$ $(1, 2)$-step compete in
$D$ if and only if $u$ (resp. $v$) has an out-neighbor different from $v$ (resp. $u$).
\end{Thm}
\begin{Thm}[\cite{choi20171}]\label{thm:competition_realizable}
 Let $n_1$ and $n_2$ be positive integers with $n_1 \geq n_2$.
 Then there exists an orientation of the complete bipartite graph $K_{n_1,n_2}$ such that any two vertices in a partite set $(1,1)$-step compete if and only if one of the following holds: (a) $n_2 =1$; (b) $n_2 \geq6$; (c) $n_1\geq10$ and $n_2=5$.
 \end{Thm}

The following is an immediate consequence of Theorem~\ref{thm:supergraph}.
\begin{Cor} \label{cor:making-(i,j)-step-multi-complete}
Let $k$ be a positive integer with $k \geq 2$; $n_1,\ldots$, and $n_k$ be positive integers such that $n_1 \geq \cdots \geq n_k$;  $n'_1, \ldots,$ and $n'_k$ be positive integers such that
$n'_1 \geq \cdots \geq n'_k$, $n'_1 \geq n_1$,
$n'_2 \geq n_2,\ldots$, and $n'_k \geq n_k$.
If the complete $k$-partite graph $K_{n_1,\ldots,n_k}$ is $(i,j)$-step competitively orientable, 
then the complete $k$-partite graph $K_{n'_1,\ldots,n'_k}$ is $(i,j)$-step competitively orientable.
\end{Cor}

\begin{Lem}\label{lem:two-in-neighbors}
Let $D$ be a digraph.
Any $(1,j)$-step common out-neighbor of two vertices in $D$ has at least two in-neighbors in $D$ for some positive integer $j$.
\end{Lem}
\begin{proof}
Suppose that a vertex $w$ is a $(1,j)$-step common out-neighbor of two vertices $u$ and $v$ in $D$. 
Since $u$ and $v$ are indistinguishable, according to the definition, we may assume that $w$ is an out-neighbor of $u$ and there exists a directed $(v,w)$-path $P_1$ of length at most $j$ not traversing $u$.
Then the vertex $x$ immediately preceding $w$ on $P_1$ and $u$ are in-neighbors of $w$.
Since $P_1$ does not traverse $u$, $x$ and $u$ are distinct and so the lemma statement follows.
\end{proof}

From now on, we write $u \to v$ to represent ``$(u,v)$ is an arc of a digraph".
  
Now we are ready to prove Theorems~\ref{thm:(1,2)-step-bipartite}, \ref{thm:tripartite-(i,j)}, and \ref{thm:(i,j)-step-complete}.

\begin{proof}[Proof of Theorem~\ref{thm:(1,2)-step-bipartite}]
We denote the partite sets of the complete bipartite graph $K_{n_1,n_2}$ by $U$ and $V$ such that 
$|U|=n_1$ and $|V|=n_2$.

{\it Case $1$}. $i+j=3$.
Without loss of generality, we may assume $i=1$ and $j=2$.
To show the ``only if" part of (a),
suppose that
there exists a $(1,2)$-step competitive
orientation $D$ of $K_{n_1,n_2}$.
Then any two vertices $u$ and $v$
$(1,1)$-step compete in $D$ if and only if $u$ and $v$ belong to the same partite set in $D$.
Thus, by Theorem~\ref{thm:competition_realizable},
 $n_2=1$, or $n_2 \geq6$, or $n_1\geq10$ and $n_2=5$.
 If $n_2=1$, then any vertex in $U$ cannot have outdegree at least $2$, which contradicts Lemma~\ref{lem:i,j-step-outdegree}.
 Thus (i) $n_2 \geq6$ or (ii) $n_1\geq10$ and $n_2=5$.
  Thus the ``only if" part is true.

Now we prove the ``if" part of (a).
By
Corollary~\ref{cor:making-(i,j)-step-multi-complete}, it is sufficient to show that  $K_{6,6}$ and $K_{10,5}$ are $(1,2)$-step competitively orientable.
We first obtain a $(1,2)$-step competitive orientation of $K_{6,6}$.
We consider the digraph $D_1$ defined by
\[V(D_1)=V(D^{\rightarrow})=V(D^{\leftarrow} )\quad \text{and} \quad A(D_1)=A(D^{\rightarrow}) \cup A(D^{\leftarrow}).\]
where $D^{\rightarrow}$ and $D^{\leftarrow}$ are the digraphs given in Figure~\ref{fig:bipartite-tournaments-complete}.
Then $D_1$ is an orientation of $K_{6,6}$ with the partite sets $\{x_1,\ldots,x_6\}$ and $\{x_7,\ldots,x_{12}\}$.
We can easily check that any two vertices in a partite set $(1,1)$-step compete and any vertex in $D_1$ has outdegree at least $2$.
Then, by Theorem~\ref{thm:distinct-(1,2)-partite},
any two vertices $u$ and $v$ belonging to distinct partite sets $(1,2)$-step compete
and so $D_1$ is $(1,2)$-step competitive.

\begin{figure}
\begin{center}
\begin{tikzpicture}[x=1.0cm, y=1.0cm]
  \vertex (x1) at (0,0) [label=left:$x_1$]{};
   \vertex (x2) at (0,1.5) [label=left:$x_2$]{};
   \vertex (x3) at (0,3) [label=left:$x_3$]{};
  \vertex (x4) at (0,4.5) [label=left:$x_4$]{};
   \vertex (x5) at (0,6) [label=left:$x_5$]{};
   \vertex (x6) at (0,7.5) [label=left:$x_6$]{};
   \vertex (x7) at (1.5,0) [label=right:$x_7$]{};
   \vertex (x8) at (1.5,1.5) [label=right:$x_8$]{};
   \vertex (x9) at (1.5,3) [label=right:$x_9$]{};
   \vertex (x10) at (1.5,4.5) [label=right:$x_{10}$]{};
   \vertex (x11) at (1.5,6) [label=right:$x_{11}$]{};
   \vertex (x12) at (1.5,7.5) [label=right:$x_{12}$]{};
  \path[->,thick,shorten >=0.5mm]
  (x6) edge (x12)
   (x6) edge (x11)
   (x6) edge (x10)
   
   (x5) edge (x9)
   (x5) edge (x11)
   (x5) edge (x10)

   (x4) edge [out=-45,in=105](x7)
   (x4) edge (x11)
   (x4) edge [out=-30,in=105] (x8)

   (x3) edge (x10)
   (x3) edge (x8)
   (x3) edge (x7)   

   (x2) edge [out=70,in=-135] (x12)
   (x2) edge (x9)
   (x2) edge (x8)     

   (x1) edge (x12)
   (x1) edge (x9)
   (x1) edge (x7)   
	;
\draw (0.75, -1) node{$D^{\rightarrow}$};
\end{tikzpicture}
\qquad \qquad 
\begin{tikzpicture}[x=1.0cm, y=1.0cm]
 \vertex (x1) at (0,0) [label=left:$x_1$]{};
   \vertex (x2) at (0,1.5) [label=left:$x_2$]{};
   \vertex (x3) at (0,3) [label=left:$x_3$]{};
  \vertex (x4) at (0,4.5) [label=left:$x_4$]{};
   \vertex (x5) at (0,6) [label=left:$x_5$]{};
   \vertex (x6) at (0,7.5) [label=left:$x_6$]{};
   \vertex (x7) at (1.5,0) [label=right:$x_7$]{};
   \vertex (x8) at (1.5,1.5) [label=right:$x_8$]{};
   \vertex (x9) at (1.5,3) [label=right:$x_9$]{};
   \vertex (x10) at (1.5,4.5) [label=right:$x_{10}$]{};
   \vertex (x11) at (1.5,6) [label=right:$x_{11}$]{};
   \vertex (x12) at (1.5,7.5) [label=right:$x_{12}$]{};
   \path[->,thick,shorten >=0.5mm]
(x12) edge (x5)
(x12) edge (x4)
(x12) edge (x3)

(x11) edge (x3)
(x11) edge (x2)
(x11) edge (x1)

(x10) edge (x4)
(x10) edge [out=245,in=30] (x2)
(x10) edge [out=255,in=15](x1)

(x9) edge [out=120,in=-35](x6)
(x9) edge (x4)
(x9) edge (x3)

(x8) edge (x6)
(x8) edge (x5)
(x8) edge (x1)

(x7) edge [out=40,in=-0](x6)
(x7) edge (x5)
(x7) edge (x2)
	;
\draw (0.75, -1) node{$D^{\leftarrow}$};
\end{tikzpicture}
\end{center}
\caption{The two digraphs mentioned in the proof of Theorem~\ref{thm:(1,2)-step-bipartite}}
\label{fig:bipartite-tournaments-complete}
\end{figure}
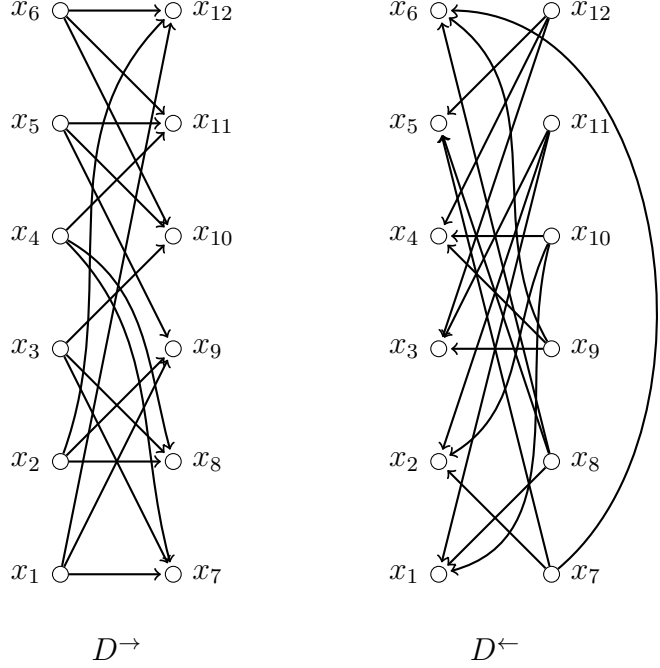

Now we obtain a $(1,2)$-step competitive orientation of $K_{10,5}$.
Let $U=\{u_1,u_2,\ldots,u_{10}\}$ and $V=\{v_1,v_2,\ldots,v_5 \}$.
We denote the subsets of $V$ with three elements by $V_1,V_2,\ldots$, and $V_{10}$.
Now we consider the bipartite tournament $D_2$ with bipartition $\{U,V\}$ and 
\[A(D_2) =\bigcup_{1\leq l \leq10}(\left \{ (u_l,v) \mid v \in V_l \} 
\cup \{ (v,u_l) \mid v \notin V_l\}   \right)
\]
This bipartite tournament is presented in the proof of Theorem 15 of \cite{choi20171} and is shown that any two vertices in a partite set $(1,1)$-step compete.
We can check that each vertex in $U$ has outdegree $3$
and each vertex in $V$ has outdegree $4$.
Then, by Theorem~\ref{thm:distinct-(1,2)-partite},
each vertex in $U$ and each vertex in $V$ $(1,2)$-step compete.
Hence $D_2$ is $(1,2)$-step competitive.
Thus we have shown that (a) is true.

{\it Case $2$}. $i+j\neq 3$. Then $i+j\geq 4$.
In this case, we will show that $K_{n_1,n_2}$ is $(i,j)$-step competitively orientable if and only if (b) or (c) holds.
To show the ``only if" part, 
    suppose that
    there exists an $(i,j)$-step competitive  orientation $D$ of $K_{n_1,n_2}$.
    Recall that $U$ and $V$ are the partite sets of $D$ such that $|U|=n_1$ and $|V|=n_2$.
 By Lemma~\ref{lem:i,j-step-outdegree}, each vertex in $U$ has at least two out-neighbors and so
    $n_2 \geq 2$. If $n_2=2$, then each vertex in $V$ is a common out-neighbor of the vertices in $U$ and so none of vertices in $V$ has an out-neighbor, which contradicts Lemma~\ref{lem:i,j-step-outdegree}.
    Thus $n_2 \geq 3$.
If $n_2 \geq 4$, then (b) holds.
Now suppose \begin{equation}\label{eq:thm:(1,2)-step-bipartite_1}n_2=3.	
 \end{equation}
Then, by Lemma~\ref{lem:i,j-step-outdegree},
\[2(n_1+3)=2(n_1+n_2)=2|V(D)|\leq \sum_{v\in V(D)}|N^+_D(v)|=|A(D)|=n_1n_2=3n_1.\]
Thus
\[n_1 \geq 6.\]
To the contrary,
suppose that $i=1$ or $j=1$.
Without loss of generality,
we may assume $i=1$.
Since $i+j \geq 4$,
$j\geq 3$.
Take two vertices $v_1$ and $v_2$ in $V$.
Then $v_1$ and $v_2$ $(1,j)$-step compete.
Let $x$ be a $(1,j)$-step common out-neighbor.
Then $x \in U$.
Therefore, by Lemma~\ref{lem:two-in-neighbors},
$x$ has at least two in-neighbors in $D$ and so, by \eqref{eq:thm:(1,2)-step-bipartite_1}, $x$ has at most one out-neighbor in $D$, which contradicts Lemma~\ref{lem:i,j-step-outdegree}. 
Hence we have shown that $i\geq 2$ and $j \geq 2$ and so (c) holds.

To show the ``if" part, suppose that (b) or (c) holds.
Then \[i+j\geq 4.\]
We let $U=\{u_1,u_2,\ldots,u_{n_1}\}$ and $V=\{v_1,v_2,\ldots,v_{n_2}\}$.
We first consider
an orientation $D_3$ of $K_{4,4}$ defined by
\[A(D_3)=\{(u_l,v_l),(u_l,v_{l+1}) \mid 1\leq l \leq 4\} \cup 
\{(v_{l},u_{l+1}),(v_{l},u_{l+2}) \mid 1\leq l \leq 4\} 
\]
(identify $u_{s+4}$ and $v_{s+4}$ with $u_s$ and $v_s$, respectively, for each $1\leq s \leq 3$).
We can check that every vertex in $D_3$ has two out-neighbors. 
Then, by Theorem~\ref{thm:distinct-(1,2)-partite}, each vertex in $U$ and each vertex in $V$ $(1,2)$-step compete.
Thus, by Proposition~\ref{prop:i,j-step_inclusion}(1), each vertex in $U$ and each vertex in $V$ $(i,j)$-step compete.
Next, we show that any two vertices in $U$ (resp.\ $V$) $(i,j)$-step compete.
It is easy to check that $u_l$ and $u_{l+1}$ (resp.\ $v_l$ and $v_{l+1}$) $(1,1)$-step compete in $D_3$ for each $1\leq l\leq 4$. 
Further, we note that 
 $u_{l}\to v_{l+1}\to u_{l+3}$ and $u_{l+2} \to v_{l+2} \to u_{l+3}$ (resp.\ $v_l \to u_{l+2} \to v_{l+3}$ and $v_{l+2} \to u_{l+3} \to v_{l+3}$) for each $1\leq l \leq 4$.
 Thus, any two vertices in $U$ (resp.\ $V$) $(2,2)$-step compete.
In addition, $u_l \to v_{l+1}$ and $u_{l+2}\to v_{l+3} \to u_{l+1} \to v_{l+1}$ (resp.\ $v_l\to u_{l+2}$ and $v_{l+2}\to u_l \to v_{l+1} \to u_{l+2}$) for each $1\leq l\leq 4$ and so any two vertices in $U$ (resp.\ $V$) $(1,3)$-step compete.
Then, by Proposition~\ref{prop:i,j-step_inclusion}(1), any two vertices in $U$ (resp.\ $V$) $(i,j)$-step compete.
Therefore $D_3$ is $(i,j)$-step competitive.

Now we consider an orientation $D_4$ of $K_{6,3}$ defined by
\begin{align*}
A(D_4)=&\{(u_{2l-1},v_l),(u_{2l},v_l),(u_{2l-1},v_{l+1}),(u_{2l},v_{l+1}) \mid 1\leq l \leq 3\}\\ & \cup \{(v_{l+2},u_{2l-1}),(v_{l+2},u_{2l}) \mid 1\leq l \leq 3\} 
\end{align*}
(identify $u_{s+6}$ and $v_{t+3}$ with $u_s$ and $v_t$, respectively, for every integers $s$ and $t$).
We can check that
any two vertices in $U$ have a common out-neighbor in $V$, that is, $(1,1)$-step compete, and
each vertex in $D_4$ has outdegree at least two.
Thus, by Theorem~\ref{thm:distinct-(1,2)-partite}, each vertex in $U$ and each vertex in $V$ $(1,2)$-step compete. 
We note that
$v_l \to u_{2l+1} \to v_{l+2}$ and 
$v_{l+1}\to u_{2l+3} \to v_{l+2}$
for each $1\leq l \leq 3$.
Thus any two vertices in $V$ $(2,2)$-step compete.
Then if $i\geq 2$ and $j\geq 2$,
we can conclude that any two vertices in $D_4$
$(i,j)$-step compete by Proposition~\ref{prop:i,j-step_inclusion}(1).
Hence we have shown $D_4$ is $(i,j)$-step competitive if $i\geq 2$ and $j\geq 2$.

Now, by applying  Corollary~\ref{cor:making-(i,j)-step-multi-complete} to $D_3$ and  $D_4$,
we may obtain an $(i,j)$-step competitive orientation of $K_{n_1,n_2}$ for (b)  $i+j \geq 4$ and $n_2 \geq 4$  and (c) $i\geq 2$, $j\geq 2$, $n_1 \geq 6$, and $n_2=3$, respectively.
Thus the ``if" part is true.
\end{proof}

\begin{proof}[Proof of Theorem~\ref{thm:tripartite-(i,j)}]
We first show the ``only if" part.
Suppose that there exists an $(i,j)$-step competitive orientation $D$ of $K_{n_1,n_2,n_3}$.
Then $D$ is nontrivial and so, by Lemma~\ref{lem:i,j-step-outdegree}, \begin{equation} \label{eq:thm:3-partite}
2|V(D)| \leq |A(D)|.
 \end{equation}
Let $V_1,V_2$, and $V_3$ be the partite sets of $D$ satisfying $|V_l|=n_l$ for each $1\leq l \leq 3$.
If $n_3 \geq 2$, then (a) holds and so we are done.
Suppose \[n_3=1.\]
If $n_2=1$, then $|V(D)|=n_1+2$ and so $|A(D)| = 2 n_1 + 1$, which contradicts \eqref{eq:thm:3-partite}.
Therefore $n_2 \geq 2$.
If $n_2 \geq 3$, then (b) holds.
Now suppose \[n_2=2.\] 
Then, by \eqref{eq:thm:3-partite},
\[2n_1+6 =2|V(D)|\leq |A(D)|=3n_1+2 \]
and so \[n_1 \geq 4.\]
Let \[V_2=\{v_1,v_2\} \quad \text{and}\quad V_3=\{v_3\}.\]
To the contrary, suppose either $i=1$ or $j=1$.
Without loss of generality, we may assume $i=1$.
Let $w_l$ be a $(1,j)$-step common out-neighbor of $v_l$ and $v_{l+1}$ in $D$ for each $1\leq l \leq 3$ (identify $v_4$ with $v_1$).
If $w_{\ell} \in V_1$ for some $\ell \in \{1,2,3\}$,
then $w_{\ell}$ has two in-neighbors by Lemma~\ref{lem:two-in-neighbors} and so $w_{\ell}$ has at most one out-neighbor in $D$, which contradicts Lemma~\ref{lem:i,j-step-outdegree}. 
Hence \[w_l \not\in V_1\] for each $1\leq l \leq 3$.
Then, since $w_l \neq v_l$ and $w_l\neq v_{l+1}$ for each $l=2,3$, 
$w_2=v_1$ and $w_3=v_2$. 
Therefore $v_3 \to v_2$ and $v_3\to v_1$.
Thus $w_1 \in V_1$, which is impossible and so $i\neq 1$.
Hence $i\geq 2$ and $j\geq 2$.
Thus (c) holds and so we have shown that the ``only if" part is true.

Now we show the ``if" part. We consider orientations $D_5$, $D_6$, and $D_7$ of
$K_{2,2,2}$, $K_{3,3,1}$, and $K_{4,2,1}$, respectively, given in Figure~\ref{fig:tripartite-tournaments-complete}.
We can check that
$D_5$ and $D_6$ are $(1,2)$-step competitive and
$D_7$ is $(2,2)$-step competitive.
Then, by applying Corollary~\ref{cor:making-(i,j)-step-multi-complete} and Proposition~\ref{prop:i,j-step_inclusion}(2) to $D_5$, $D_6$, and $D_7$, 
we may obtain an $(i,j)$-step competitive orientation of $K_{n_1,n_2,n_3}$ for (a) $n_3\geq 2$; (b) $n_2\geq 3$ and $n_3=1$; (c) $i\geq 2$, $j\geq 2$, $n_1 \geq 4$, $n_2=2$, and $n_3=1$, respectively.
\end{proof}

\begin{proof}[Proof of Theorem~\ref{thm:(i,j)-step-complete}]
We first show the ``only if" part.
Suppose that there exists an $(i,j)$-step competitive orientation $D$ of $K_{n_1,n_2,\ldots,n_k}$.
If $k \geq 5$, then (b) holds.
Suppose $k=4$.
If $n_2\geq 2$, then (ii) of (a) holds.
Now suppose $n_2=1$.
Then $n_3=n_4=1$ and so $|V(D)|=n_1+3$ and
$|A(D)|=3n_1+3$.
Thus, by Lemma~\ref{lem:i,j-step-outdegree}, \[2n_1+6=2|V(D)| \leq |A(D)|=3n_1+3\] and so
$n_1 \geq 3$.
Therefore (i) of (a) holds. Thus we have shown that the ``only if" part is true.

Now we show the ``if" part.
We consider $(1,2)$-step competitive orientations $D_8$, $D_9$, and $D_{10}$ of  $K_{3,1,1,1}$, $K_{2,2,1,1}$, and $K_{1,1,1,1,1}$, respectively, given in Figure~\ref{fig:tournaments-complete}.
By applying Corollary~\ref{cor:making-(i,j)-step-multi-complete} to $D_8$, $D_9$, and $D_{10}$,
we may obtain a $(1,2)$-step competitive orientation of $K_{n_1,n_2,\ldots,n_k}$ when
(a) $k=4$, and (i) $n_1 \geq 3$ and $n_2=1$ or (ii) $n_2 \geq 2$;
(b) $k\geq 5$, respectively.
Then, by Proposition~\ref{prop:i,j-step_inclusion}(2),
the ``if'' part is true.
\end{proof}

\begin{figure}
\begin{center}
\begin{tikzpicture}[x=1.0cm, y=1.0cm]
   \vertex (x1) at (0,0) [label=above:$$]{};
   \vertex (x2) at (0,1.5) [label=above:$$]{};
   \vertex (x3) at (1.5,0) [label=above:$$]{};
   \vertex (x4) at (1.5,1.5) [label=above:$$]{};
   \vertex (x5) at (3,0) [label=above:$$]{};
   \vertex (x6) at (3,1.5) [label=above:$$]{};
   \path[->,thick]
   (x1) edge (x3)
   (x3) edge  (x5)
   (x5) edge  [out=-135,in=-45] (x1)
   (x4) edge (x2)
   (x6) edge  (x4)
   (x2) edge [out=45,in=135] (x6)
   (x2) edge  [out=-25,in=140] (x5)
   (x4) edge  (x1)
   (x6) edge (x3)
   (x1) edge (x6)
   (x3) edge [out=135,in=-55] (x2)
   (x5) edge [out=120,in=-35] (x4)
	;
\draw (1.5, -1.5) node{$D_5$};
\end{tikzpicture}
\qquad \qquad 
\begin{tikzpicture}[x=1.0cm, y=1.0cm]
   \vertex (x1) at (0,0) [label=above:$$]{};
   \vertex (x2) at (0,1.5) [label=above:$$]{};
   \vertex (x3) at (0,3) [label=above:$$]{};
   \vertex (x4) at (1.5,0) [label=above:$$]{};
   \vertex (x5) at (1.5,1.5) [label=above:$$]{};
   \vertex (x6) at (1.5,3) [label=above:$$]{};
   \vertex (x7) at (3,3) [label=above:$$]{};
   \path[->,thick]
   (x3) edge [out=45,in=135] (x7)
   (x3) edge (x5)
   (x2) edge (x6)
   (x2) edge [out=30,in=-160] (x7)
   (x1) edge (x5)
   (x1) edge (x4)
   (x6) edge (x3)
   (x6) edge (x1)
   (x5) edge (x7)
   (x5) edge  (x2)
   (x4) edge (x2)
   (x4) edge  (x3)
   (x7) edge  (x4)
   (x7) edge (x6)
   (x7) edge  [out=-90,in=20] (x1)
	;
\draw (1.5, -1) node{$D_6$};
\end{tikzpicture}
\qquad \qquad 
\begin{tikzpicture}[x=1.0cm, y=1.0cm]
   \vertex (x1) at (0,0) [label=above:$$]{};
   \vertex (x2) at (0,1.5) [label=above:$$]{};
   \vertex (x3) at (0,3) [label=above:$$]{};
   \vertex (x4) at (0,4.5) [label=above:$$]{};
   \vertex (x5) at (1.5,0.75) [label=above:$$]{};
   \vertex (x6) at (1.5,3.75) [label=above:$$]{};
   \vertex (x7) at (2.5,2.25) [label=above:$$]{};
   \path[->,thick]
   (x1) edge [out=-45,in=270] (x7)
   (x1) edge (x6)
   (x2) edge (x6)
   (x2) edge (x7)
   (x3) edge (x5)
   (x3) edge (x7)
   (x4) edge (x5)
   (x4) edge [out=45,in=90] (x7)
   (x5) edge (x1)
   (x5) edge (x2)
   (x6) edge (x3)
   (x6) edge (x4)
   (x7) edge (x5)
   (x7) edge (x6) 
	;
\draw (1.5, -1) node{$D_7$};
\end{tikzpicture}
\end{center}
\caption{ The digraphs $D_5$ and $D_6$ are $(1,2)$-step competitive orientations of $K_{2,2,2}$ and $K_{3,3,1}$, respectively, and the digraph $D_7$ is a $(2,2)$-step competitive orientation of $ K_{4,2,1}$.}
\label{fig:tripartite-tournaments-complete}
\end{figure}
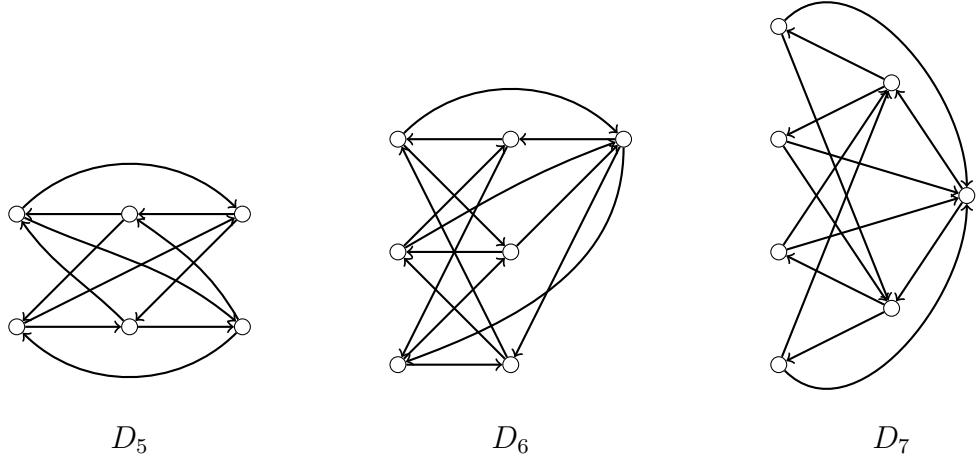

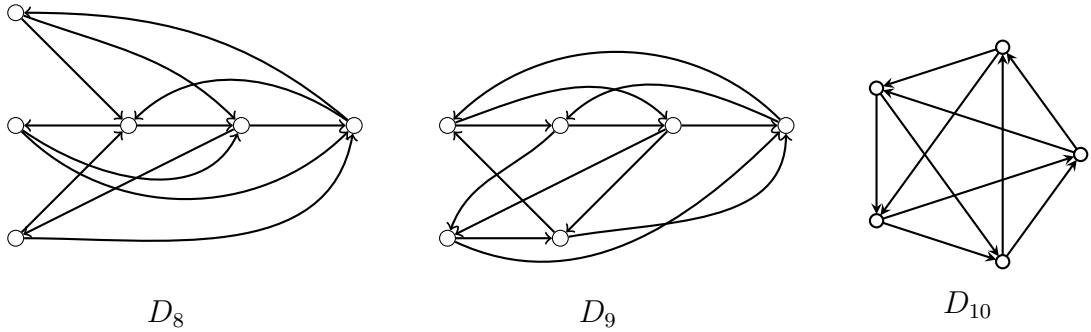
\begin{figure}
\begin{center}
\begin{tikzpicture}[x=1.0cm, y=1.0cm]
   \tikzset{middlearrow/.style={decoration={
  markings,
  mark=at position 0.95 with
  {\arrow{#1}},
  },
  postaction={decorate}
  }
  }

   \vertex (x1) at (0,0) [label=above:$$]{};
   \vertex (x2) at (0,1.5) [label=above:$$]{};
   \vertex (x3) at (0,3) [label=above:$$]{};
   \vertex (x4) at (1.5,1.5) [label=above:$$]{};
   \vertex (x5) at (3,1.5) [label=above:$$]{};
   \vertex (x6) at (4.5,1.5) [label=above:$$]{};
 \path[->,thick]
   (x3) edge  (x4)
  (x3) edge [out=-20,in=135](x5)
   (x2) edge [out=-35,in=-115](x5)
   (x2) edge [out=-45,in=-135](x6)
   (x1) edge (x4)
   (x1) edge [out=0,in=-100] (x6)
   (x4) edge (x2)
   (x4) edge (x5)
   (x5) edge (x1)
   (x5) edge (x6)
   (x6) edge [out=145,in=45] (x4)
   (x6) edge [out=140,in=0] (x3)
	;
\draw (2, -1) node{$D_8$};
\end{tikzpicture}
\qquad 
\begin{tikzpicture}[x=1.0cm, y=1.0cm]
   \vertex (x1) at (0,0) [label=above:$$]{};
   \vertex (x2) at (0,1.5) [label=above:$$]{};
   \vertex (x3) at (1.5,0) [label=above:$$]{};
   \vertex (x4) at (1.5,1.5) [label=above:$$]{};
   \vertex (x5) at (3,1.5) [label=above:$$]{};
   \vertex (x6) at (4.5,1.5) [label=above:$$]{};
   \path[->,thick]
   (x2) edge (x4)
   (x2) edge [out=20,in=135] (x5)
   (x1) edge (x3)
   (x1) edge [out=-25,in=-135] (x6)
   (x3) edge (x2)
   (x3) edge [out=10,in=-90] (x6)
   (x4) edge [out=-135,in=70](x1)
   (x4) edge (x5)
   (x5) edge (x3)
   (x5) edge (x6)
   (x5) edge (x1)
   (x6) edge [out=135,in=45](x2)
   (x6) edge [out=155,in=45](x4)
	;
\draw (2, -1) node{$D_9$};
\end{tikzpicture}
\qquad
 \begin{tikzpicture}[auto,thick]
       \tikzstyle{player}=[minimum size=5pt,inner sep=0pt,outer sep=0pt,draw,circle]

    \tikzstyle{player1}=[minimum size=2pt,inner sep=0pt,outer sep=0pt,fill,color=black, circle]
    \tikzstyle{source}=[minimum size=5pt,inner sep=0pt,outer sep=0pt,ball color=black, circle]
    \tikzstyle{arc}=[minimum size=5pt,inner sep=1pt,outer sep=1pt, font=\footnotesize]
    \path (0:1.5cm)+(0,1)     node [player]  (v0)
    {};
    \path (72:1.5cm)  +(0,1)   node [player]  (v1)  {};
    \path (144:1.5cm) +(0,1)    node [player]  (v2){};
    \path (216:1.5cm) +(0,1)    node [player]  (v3){};
 \path (288:1.5cm)   +(0,1)  node [player]  (v4){};

  \draw[black,thick,-stealth] (v0) - +(v1);
  \draw[black,thick,-stealth] (v0) - +(v2);

   \draw[black,thick,-stealth](v1) - +(v2);
  \draw[black,thick,-stealth] (v1) - +(v3); 
  
   \draw[black,thick,-stealth](v2) - +(v3);
  \draw[black,thick,-stealth] (v2) - +(v4);
  
  \draw[black,thick,-stealth](v3) - +(v4);
  \draw[black,thick,-stealth] (v3) - +(v0);
  
  \draw[black,thick,-stealth](v4) - +(v0);
  \draw[black,thick,-stealth] (v4) - +(v1);
  \draw (0, -1) node{$D_{10}$};
    \end{tikzpicture}
\end{center}
\caption{ 
The digraphs $D_8$, $D_9$, and $D_{10}$ are $(1,2)$-step competitive orientations of $K_{3,1,1,1}$, $K_{2,2,1,1}$, and $K_{1,1,1,1,1}$, respectively.}
\label{fig:tournaments-complete}
\end{figure}
\section*{Declaration of competing interest}
The authors declare that they have no known competing financial interests or personal relationships that could have
appeared to influence the work reported in this paper.
\section*{Data availability}
No data was used for the research described in the article.

\section*{Acknowledgement}
This research was supported by the National Research Foundation of Korea(NRF) grant funded by the Korea government(MSIT) (NRF-2022R1A2C1009648 and 2016R1A5A1008055).
\bibliographystyle{plain}

\end{document}